\RequirePackage[T1]{fontenc}
\documentclass[preprint,twoside,12pt]{amsart}

\usepackage{amsfonts}
\usepackage{amsmath}
\usepackage{amsthm}
\usepackage{amssymb}
\usepackage{bbm}
\usepackage{fancyhdr}
\usepackage{tikz}
\usetikzlibrary{positioning,arrows.meta}
\usepackage{etoolbox}
\usepackage{stmaryrd}
\usepackage[all]{xy}
\usepackage{physics}
\usepackage{tikz-cd}
\usepackage{mathrsfs}
\usepackage{mathtools}
\usepackage{xcolor}
\usepackage{latexsym}
\usepackage[normalem]{ulem}
\usepackage{thm-restate}
\usepackage{bussproofs}

\tikzcdset{row sep/normal=50pt, column sep/normal=50pt}

\usepackage[bookmarks=true,bookmarksdepth=3,bookmarksnumbered=true,colorlinks=false,pdfborder={0,0,0}]{hyperref} 

\makeatletter
\newcommand{\SafeTocLink}[3] {
\hyperlink{#1}{#2\hfill #3}
\patchcmd{\l@section}
{\@dottedtocline{1}{0em}{1.5em}{#1}{#2}}{\@dottedtocline{1}{0em}{1.5em}{\SafeTocLink[\@currentHref]}
{#1}{#2}}{} }{}{}
\makeatother

\theoremstyle{definition}
\newtheorem{theorem}{Theorem}[section]

\newtheorem{lemma}[theorem]{Lemma}
\newtheorem{corollary}[theorem]{Corollary}
\newtheorem{proposition}[theorem]{Proposition}
\newtheorem{definition}[theorem]{Definition}

\newtheorem{remark}[theorem]{Remark}

\newcommand{\TR}{\mathstrut \mathrm{tr} }
\newcommand{\Mor}{\mathstrut \mathrm{Mor} }
\newcommand{\id}{\mathstrut \mathrm{id} }
\newcommand{\colim}{\mathstrut \mathrm{colim} }
\newcommand{\Hom}{\mathstrut \mathrm{Hom} }
\newcommand{\Fix}{\mathstrut \boldsymbol{\mathrm{Fix}} }
\newcommand{\dom}{\mathstrut \boldsymbol{\mathrm{dom}}\,}
\newcommand{\cod}{\mathstrut \boldsymbol{\mathrm{cod}}\,}
\newcommand{\Stab}{\mathstrut \boldsymbol{\mathrm{Stab}} }
\newcommand{\Tri}{\mathstrut \boldsymbol{\mathrm{Tri}} }

\newcommand{\res}{\mathstrut \boldsymbol{\mathrm{Res}} }
\newcommand{\Ind}{\mathstrut \boldsymbol{\mathrm{Ind}} }
\newcommand{\incl}{\mathstrut \boldsymbol{\mathrm{Incl}} }
\newcommand{\Iso}{\mathstrut \boldsymbol{\mathrm{Iso}} }
\newcommand{\set}{\mathstrut \boldsymbol{\mathrm{set}} }
\newcommand{\vect}{\mathstrut \boldsymbol{\mathrm{vect}} }

\newcommand{\Pair}{\mathstrut \boldsymbol{\mathrm{Pair}} }

\newcommand{\coloreduline}[3][red]{
\tikz[baseline=(X.bese)]{\node[innee sep=0pt, outer sep=0pt](X) {\textcolor{#2}{#3}};\draw[#1, line width=0.6pt](X.south west) -- (X.south east); } }

\newcommand{\Thmref}[1]{\hyperref[#1] {\coloreduline[red]{black}{\ref*{#1}}}}

\title{\textbf{DOUBLE COSET FOR GROUPOIDS}}

\author{KEITARO SHIIZUKA}

\address{Department of Mathematics, Kindai University, 3-4-1, Kowakae, Higashi-Osaka, Osaka 577-8502, Japan}
\email{keitaro.shizuka.math@gmail.com}

\begin{document}

\maketitle

\begin{abstract}
We investigate the double cosets of a groupoid, focusing primarily on their enumeration, by means of two different approaches. The first approach extends the Cauchy-Frobenius lemma to groupoids and interprets it in terms of groupoid actions. The second approach is based on linear representations of a groupoid arising from its action on a set of functions.
\end{abstract}


\tableofcontents

\section{Introduction}
Groupoids were introduced by H. Brandt in 1927 (see \cite{Bra27}) as  a generalization of groups. Unlike groups, a groupoid does not require that every pair of elements be composable. Many fundamental concepts and theorems in group theory such as Cayley's theorem, Lagrange's theorem, Sylow's theorem, semidirect products, centers, commutators, and solvability, have their counterparts in the context of groupoid (see \cite{Iva02}, \cite{AM20}, \cite{MP22}, \cite{BGLPT23}).
There are several equivalent ways to define groupoids and their actions. In this manuscript, we adopt a categorical framework, viewing groupoids as categories in which all morphisms are invertible. From this standpoint, an action of a groupoid $\mathcal{G}$ is described by a functor $\mathcal{G}\to \set$.
In Section $2$, we provide notations and basic definitions. In Section 3, we develop an extension of Cauchy-Frobenius theorem in the context of groupoids. In group theory, let $G$ be a group, $H,K\le G$ and $g\in G$. The double cosets $HgK$ appear as orbits of certain actions of the direct product $H\times K$ on $G$. In Section 3, we extend this viewpoint to groupoids. We interpret $\mathcal{H}g\mathcal{K}$ as an orbit of a groupoid action of the direct product groupoid $\mathcal{H}\times\mathcal{K}$. We then show that the action groupoid arising from this action is isomorphic to a comma category $\incl^\mathcal{G}_\mathcal{K}\downarrow \incl^\mathcal{G}_\mathcal{H}$, and establish the following theorem concerning the size of $\mathcal{H}g\mathcal{K}$. 
\setcounter{section}{3}
\setcounter{theorem}{3}
\begin{theorem}
Let $\mathcal{H}$ and $\mathcal{K}$ be wide subgroupoids of a connected groupoid $\mathcal{G}$.
Let $\mathcal{H}\cong\coprod_{\lambda} \mathcal{H}_{\lambda}$ and $\mathcal{K}\cong\coprod_{\mu} \mathcal{K}_{\mu}$ be their decompositions into connected components.
For $g\in\mathcal{G}_1$, let $\mathcal{H}_\lambda$ and $\mathcal{K}_\mu$ be the connected components containing  $\cod g$ and $\dom g$, respectively. Then,
\[|\mathcal{H}g\mathcal{K}|=\frac{|(\mathcal{H}_\lambda)_0||(\mathcal{K}_\mu)_0||\mathcal{H}_{\cod g}||\mathcal{K}_{\dom g}|}{|\mathcal{H}_{\cod g}\cap{}^g\mathcal{K}_{\dom g}|}.\]
holds.
\end{theorem}
\setcounter{section}{0}

In Section 4 and 5, we approach the problem from the viewpoint  of the linear representation theory of groupoids. Attempts to extend the representation theory of finite groups to that of groupoids have been studied by Barbar\'{a}n S\'{a}nchez et al. \cite{BE19} and Ibort et al. \cite{IR19}. In the representation theory of groups, it is a well-known fact that the permutation module $\mathbb{C}[G/H]$ can be identified with the space of functions $f:G\to\mathbb{C}$ that are constant on the left cosets of $H$. In this manuscript, we present an attempt to generalize the fact from group theory to the setting of groupoids. We prove that the representation 
$\mathbb{C}[\mathcal{G}/\mathcal{H}]$ in the representation theory of groupoids can be expressed as an induction from the trivial representation via Kan extensions, we also prove the following theorem. 
\setcounter{section}{5}
\setcounter{theorem}{3}
\begin{theorem}
Let $\mathcal{H}$ and $\mathcal{K}$ be connected wide subgroupoids of $\mathcal{G}$. Then there exists an isomorphism of $\mathbb{C}$-vector spaces \[\mathbb{C}_{\mathcal{H}\backslash\mathcal{G}/\mathcal{K}}\cong \vect^\mathcal{G}(Y_\mathcal{H},Y_\mathcal{K}).\]
\end{theorem}
The following result holds as a corollary to the theorem above. \begin{corollary}
Let $\mathcal{H}$ and $\mathcal{K}$ be connected wide subgroupoids of $\mathcal{G}$. 
Then, 
\[
|\mathcal{H}\backslash\mathcal{G}/\mathcal{K}|=\langle \chi_{\Ind^\mathcal{G}_\mathcal{H}\Tri},\chi_{\Ind^\mathcal{G}_\mathcal{K}\Tri} \rangle\]
holds.
\end{corollary}

\setcounter{section}{1}

\section{Preliminaries}
\subsection{Notations and conventions}
\begin{itemize}
\item In this manuscript, we use calligraphic notation (such as $\mathcal{G}$) to represent finite groupoids. By contrast, standard symbols (such as $G$) denote either objects within a groupoid or groupoids with a single object (that is, groups).
\item A finite group $G$ is identified with its associated one-object groupoid. \item For a groupoid $\mathcal{G}$, the composition $g\circ g'$ of morphisms is written more simply as $gg'$, which serves as a convenient notational simplification when treating groupoids as extensions of groups. \item In the disjoint union $\coprod_{i\in I} A_i$ of the family of sets $\{A_i\}_{i\in I}$, the element corresponding to $x\in A_i$ is denoted by $\langle i,x \rangle$.
\item The empty category is a finite groupoid, but we exclude it from the class of finite groupoids considered in this manuscript.
\item Let $\mathfrak{U}$ be a Grothendieck universe. Throughout this manuscript, we assume that every category $\mathcal{C}$ is $\mathfrak{U}$-small, meaning that the collection of objects $\mathcal{C}_0$ and the collection of morphisms $\mathcal{C}_1$ both belong to $\mathfrak{U}$.
\item For a category $\mathcal{C}$, the maps $\dom$ and $\cod$ are defined as follows:
\[\dom:\mathcal{C}_1\to\mathcal{C}_0\;;\;f\mapsto \dom f,\]
\[\cod:\mathcal{C}_1\to\mathcal{C}_0\;;\;f\mapsto \cod f.\]
\item Let $\mathcal{C},\mathcal{D}$ and $\mathcal{E}$ be categories and $F:\mathcal{C}\to \mathcal{D}$ a functor. Then $F$ induces a functor $F^*:\mathcal{E}^\mathcal{D}\to \mathcal{E}^\mathcal{C}$, given by $G\mapsto G\circ F$. We denote by $F^\dagger$ the left adjoint of $F^*$. In this manuscript, the functors $F^\dagger$ are constructed via Kan extensions.
\item $|X|$ denotes the cardinality of a set $X$.
\item $\set$ denotes the category of finite sets and $\vect$ denotes the category of finite dimensional vector spaces over $\mathbbm{C}$.
\item We let $\cong$ denote an isomorphism of categories.
\end{itemize}

\subsection{Groupoids}
In this subsection, we prepare the groundwork for the subsequent sections. Throughout this paper, we treat groupoids within a categorical framework. Algebraic approaches to the topics discussed in this section can be found in \cite{IR19}, \cite{MP22}, \cite{BE19}, \cite{AM20}.

\begin{itemize}
\item A category $\mathcal{G}$ is called a \textbf{groupoid} if every morphism in $\mathcal{G}$ is invertible. 
\item For each object $x\in\mathcal{G}_0$, the set $\mathcal{G}(x,x)$, denoted by $\mathcal{G}_x$, is referred to as the \textbf{isotropy group} at $x\in\mathcal{G}_0$. 
\item The \textbf{group bundle} $\Iso\;\mathcal{G}$ is the groupoid defined by the following conditions: 
\begin{enumerate}
\item $(\Iso\;\mathcal{G})_0:=\mathcal{G}_0,$
\item \[\Iso\;\mathcal{G}(G,G'):=\begin{cases}\mathcal{G}(G,G)&\text{if $G=G'$,}\\\emptyset&\text{if $G\neq G'$.}\end{cases}\]
\item with composition induced from $\mathcal{G}$.
\end{enumerate}

\item A \textbf{subgroupoid} $\mathcal{H}$ of a groupoid $\mathcal{G}$ is a subcategory of $\mathcal{G}$ that is itself a groupoid. If $\mathcal{H}$ contains all objects of $\mathcal{G}$, $\mathcal{H}$ is called a \textbf{wide subgroupoid} of $\mathcal{G}$. For a subgroupoid $\mathcal{H}$ of $\mathcal{G}$, then the inclusion functor is denoted $\incl^\mathcal{G}_\mathcal{H}$. 

\item For $G\in \mathcal{G}_0$, define  \[\Mor_\mathcal{G}(G,-):=\coprod_{G'\in\mathcal{G}_0}\mathcal{G}(G,G'),\;\; \Mor_\mathcal{G}(-,G):=\coprod_{G'\in\mathcal{G}_0}\mathcal{G}(G',G).\]

\item For a finite set $X$, the \textbf{pair groupoid} $\Pair(X)$ is defined by the following conditions: 
\begin{enumerate}
\item object: $\Pair(X)_0:=X,$
\item morphism: $\Pair(X)(x,y):=\{(x,y)\}$, for \;$x,y\in X,$
\item composition: The composition of $f:=(x,y)$ and $g:=(y,z)$ is defined by $gf:=(x,z)$.
\end{enumerate}
\end{itemize}

\begin{definition}
Let $\mathcal{G}$ be a groupoid. Define a relation on $\mathcal{G}_0$ by declaring $x\sim y$ if there exists a morphism $g:x\to y$. This is an equivalence relation. The quotient set $\mathcal{G}_0/{\sim}$, denoted by $\pi_0(\mathcal{G})$, is called the set of \textbf{connected components} of $\mathcal{G}$. If $\pi_0(\mathcal{G})$ consists of a single element, then $\mathcal{G}$ is said to be a \textbf{connected groupoid}.
\end{definition}

The following proposition is well known.
\begin{proposition}\label{gpdst}
Let $\mathcal{G}$ be a connected groupoid and let $x\in\mathcal{G}_0$. Then, $\mathcal{G}\cong\mathcal{G}_x\times\;\Pair(\mathcal{G}_0)$.
\end{proposition}

\begin{lemma}\label{re0}
Let $\mathcal{G}$ be a connected groupoid.  If $x\in\mathcal{G}_0$, then $|\mathcal{G}(y,z)|=|\mathcal{G}_x|$ for all objects $y,z$.
\end{lemma}

\begin{proof}
Since $\mathcal{G}$ is connected, we can choose morphisms $f_y:x\to y$ and $f_z:x\to z$. Consider the maps
\[\begin{array}{rccc}
f:&\mathcal{G}_x&\longrightarrow &
\mathcal{G}(x,y)\\
 &\rotatebox{90}{$\in$}& &\rotatebox{90}{$\in$}\\
 &u&\longmapsto&f_yu
\end{array},
\]
\[\begin{array}{rccc}
g:&\mathcal{G}(x,y)&\longrightarrow &
\mathcal{G}(y,z)\\
 &\rotatebox{90}{$\in$}& &\rotatebox{90}{$\in$}\\
 &a&\longmapsto&f_za^{-1}
\end{array}.
\]
Since the above map is bijective, the claim follows.
\end{proof}

\begin{itemize}
\item Let $\mathcal{G}$ be a groupoid, and let $\mathcal{H}$ and $\mathcal{K}$ be subgroupoids of $\mathcal{G}$. For $g\in \mathcal{G}_1$, the $(\mathcal{H},\mathcal{K})$-double coset of $g$ is the set $\mathcal{H}g\mathcal{K}:=\{hgk\mid h\in\mathcal{H}_1,\, k\in\mathcal{K}_1,\cod k=\dom g,\,\cod g=\dom h\}$. The set of all $(\mathcal{H},\mathcal{K})$-double cosets in $\mathcal{G}$ is denoted by $\mathcal{H}\backslash \mathcal{G}/\mathcal{K}$. 

\item The functor category $\set^\mathcal{G}$ is called the \textbf{category of $\mathcal{G}$-sets}, and its objects are simply referred to as \textbf{$\mathcal{G}$-set}. 

\item For a wide subgroupoid $\mathcal{H}$ of $\mathcal{G}$ and $x\in\mathcal{G}_0$, define a relation $\sim_{\mathcal{H},x}$ on $\Mor_\mathcal{G}(-,x)$ by 
\[a\sim_{\mathcal{H},x} b \iff  a^{-1}b\in \mathcal{H}_1\] 
for all $a,b\in \Mor_\mathcal{G}(-,x)$.
\item If $\mathcal{H}$ is a wide subgroupoid of $\mathcal{G}$, then the functor $\mathcal{G}/\mathcal{H}:\mathcal{G}\to \set$ is defined by the following conditions: 
\begin{enumerate}
\item For an object $x\in \mathcal{G}_0$, define $\mathcal{G}/\mathcal{H}(x):=\Mor_\mathcal{G}(-,x)/{\sim_{\mathcal{H},x}}$.
\item For a morphism $g:x\to y$, define $\mathcal{G}/\mathcal{H}(g):\Mor_\mathcal{G}(-,x)/{\sim_{\mathcal{H},x}}\to
\Mor_\mathcal{G}(-,y)/{\sim_{\mathcal{H},y}}$ by $k\mathcal{H}\mapsto gk\mathcal{H}$.
\end{enumerate}

\end{itemize}

The following proposition is well known.

\begin{proposition}\cite[Theorem 4.4]{BGLPT23}
Let $\mathcal{H}$ be a subgroupoid of a connected groupoid $\mathcal{G}$ with connected components denoted by $\mathcal{H}_i$ and, for $e_i\in (\mathcal{H}_i)_0$, the isotropy group $H_{e_i}$ of $\mathcal{H}_i$ for $1\le i\le n$. Then \[(\mathcal{G}:\mathcal{H})=|\mathcal{H}_0|(\sum^n_{i=1}(\mathcal{G}_{e_i}:\mathcal{H}_{e_i})).\]
\end{proposition}

Moreover, many of results in this paper provide alternative descriptions of $(\mathcal{G}:\mathcal{H})$.

\begin{itemize}
\item A functor $X:\mathcal{G}\to \vect$ is referred to as a \textbf{linear representation of $\mathcal{G}$}. In this case, the $\mathbbm{C}$-linear space $\bigoplus_{G\in\mathcal{G}_0} X(G)$ is called the \textbf{representation space} of $X$. The functor category $\vect^\mathcal{G}$ is called \textbf{the category of linear representations of $\mathcal{G}$}.

\item If $\mathcal{G}$ is a groupoid, then the trivial representation of $\mathcal{G}$ is the functor $\Tri:\mathcal{G}\to \vect$ defined by the following conditions: 
\begin{enumerate}
\item For an object $x\in \mathcal{G}_0$, define $\Tri(x):=\mathbbm{C}$.
\item For a morphism $g:x\to y $, define $\Tri(g):=\id_{\mathbbm{C}}$.
\end{enumerate}

\end{itemize}

\begin{definition}
Let $X$ be a $\mathcal{G}$-set. The \textbf{action groupoid} $\mathcal{G}\ltimes X$ is the groupoid defined by the following conditions: 
\begin{enumerate}
\item Its objects are given by $(\mathcal{G}\ltimes X)_0:=\coprod_{G\in \mathcal{G}_0} X(G).$
\item For objects $\langle G,x \rangle$ and $\langle G',x' \rangle$, the morphisms are  
\[(\mathcal{G}\ltimes X)(\langle G,x \rangle, \langle G',x' \rangle):=\{g\in\mathcal{G}(G,G')\mid X(g)(x)=x'\}.\]
\item Composition is inherited from $\mathcal{G}$.
\end{enumerate}
For $\langle G,x \rangle\in \coprod_{G\in \mathcal{G}_0} X(G)$, the \textbf{stabilizer} at $x$ is defined by $\Stab(x):=(\mathcal{G}\ltimes X)_{\langle G,x \rangle}$. A $\mathcal{G}$-set $X$ is said to be \textbf{transitive} if the $\mathcal{G}\ltimes X$ is connected.
\end{definition}

\begin{remark}\label{E}
The action groupoid $\mathcal{G}\ltimes X$ coincides with the category of elements of $X$ (the Grothendieck construction).
\end{remark}

\begin{lemma}\label{re1}
Let $F:\set\to\vect$ be the linerization functor assigning to each finite set $X$ the free complex vector space $\mathbb{C}[X]$. Then the functors $\Tri$ and $F\circ \triangle 1$ are naturally isomorphic.
\end{lemma}

\begin{proof}
This can be verified directly.
\end{proof}

\begin{itemize}
\item  Let $X$ be a $\mathcal{G}$-set. The \textbf{permutation representation} associated with $X$ is the functor  $\mathbb{C}[X]:\mathcal{G}\to \vect$ defined by the following conditions: 
\item For an object $x\in \mathcal{G}_0$, define $\mathbb{C}[X](x):=\mathbb{C}[X(x)]$
\item For a morphism $g:x\to y$, define $\mathbb{C}[X](g):\mathbb{C}[X(x)]\to \mathbb{C}[X(y)]$ by extending $s\mapsto X(g)(s)$ $\mathbb{C}$-linearly for $s\in X(x)$.

\item For an inclusion functor $\incl^\mathcal{G}_\mathcal{H}$, define functors $\res^\mathcal{G}_\mathcal{H}:\mathcal{H}\to\vect$ and $\Ind^\mathcal{G}_\mathcal{H}:\mathcal{G}\to\vect$ by $\res^\mathcal{G}_\mathcal{H}:=(\incl^\mathcal{G}_\mathcal{H})^*$, $\Ind^\mathcal{G}_\mathcal{H}:=(\incl^\mathcal{G}_\mathcal{H})^\dagger$.

\item Let $R$ be a linear representation of $\mathcal{G}$, and let $V$ denote its representation space. For each morphism $g:x\to y$ in $\mathcal{G}$, let $\hat{R}(g):V\to V$ be the linear map induced by $R(g):R(x)\to R(y)$. The \textbf{character} of $R$ $\chi_R:\mathcal{G}_1\to\mathbb{C}$ is the map defined by
\[g\mapsto \TR\;(\hat{R}(g)).\]

\item If $\chi_R$ and $\chi_{R'}$ are characters of $R$ and $R'$, then the inner product of $\chi_R$ and $\chi_{R'}$ is defined as follows.
$\langle \chi_R,\chi_{R'} \rangle$ is the sum \[\frac{1}{|\mathcal{G}_0|}\sum_{x\in\mathcal{G}_0}\langle {\chi_{R\circ\incl^\mathcal{G}_{\mathcal{G}_x}}},{\chi_{R'\circ\incl^\mathcal{G}_{\mathcal{G}_x}}} \rangle.\]
\end{itemize}

\begin{lemma}\cite[Theorem 4]{IR19}
Let $R,R'$ be representations of $\mathcal{G}$, and let $V$ and $V'$ be the corresponding representation spaces. Then 
\[\dim \Hom_{\mathcal{G}}(V,V')=\langle \chi_R,\chi_{R'} \rangle\] holds.\end{lemma}

\begin{corollary}\label{product}
Let $R,R'$ be representations of $\mathcal{G}$, and let $V$ and $V'$ be the corresponding representation spaces. Then 
\[\Hom_{\mathcal{G}}(V,V')=\vect^\mathcal{G}(R,R')\] 
holds. In particular, \[\dim \Hom_{\mathcal{G}}(V,V')=\dim \vect^\mathcal{G}(R,R')\] holds.\end{corollary}

\begin{proof}
Let $\phi:R\to R'$ be a natural transformation, and let $\hat{\phi}:V\to V'$ denote the induced $\mathcal{G}$-linear map. Then the correspondence $\phi\mapsto \hat{\phi}$ yields an isomorphism of vector spaces $\vect^\mathcal{G}(R,R')\cong \Hom_\mathcal{G}(V,V')$.
\end{proof}

\section{An extension of Cauchy-Frobenius lemma to groupoids}
In this section, we prove the Cauchy-Frobenius theorem for groupoids and develop an approach to double cosets using groupoid actions.
\begin{theorem}\label{CF}
Let $\mathcal{G}$ be a connected groupoid and $X$ be a $\mathcal{G}$-set. For $g\in(\Iso\;\mathcal{G})_1$, we set \[\Fix(g):=\{x\in X(\dom g)\mid X(g)(x)=x\}.\] Then
\[|\pi_0{(\mathcal{G}\ltimes X)}|=\frac{1}{|(\Iso\;\mathcal{G})_1|}\sum_{g\in(\Iso\;\mathcal{G})_1}|\Fix(g)|.\]
\end{theorem}

\begin{proof}
Let $X$ be a $\mathcal{G}$-set and let 
\[D:=\{(g,x)\in (\Iso\;\mathcal{G})_1\times X(\dom g)\mid X(g)(x)=x\}.\] Suppose that $|\pi_0{(\mathcal{G}\ltimes X)}|=n$. Let $\mathcal{G}\ltimes X \cong\coprod^n_{i=1}(\mathcal{G}\ltimes X)_{[i]}$ be the decomposition of $\mathcal{G}\ltimes X$ into its connected components. We compute $|D|$ in two different ways. \\$(1)$ \[|D|=\smashoperator[r]{\sum_{g\in (\Iso\;\mathcal{G})_1}}|\Fix(g)|.\]
$(2)$ \begin{center}
$\begin{aligned}|D|&=\sum_{G\in\mathcal{G}_0}\smashoperator[r]{\sum_{x\in X(G)}}|\Stab(x)|\\&=\sum_{G\in\mathcal{G}_0}\sum_{i=1}^n\smashoperator[r]{\sum_{x\in (\mathcal{G}\ltimes X)_{[i]}\cap X(G)}}|\Stab(x)|& &\\
&\overset{(*)}{=}\sum_{G\in\mathcal{G}_0}\sum_{i=1}^n|\mathcal{G}_G|\\
&=|(\Iso\;\mathcal{G})_1|n.& &\label{B}
\end{aligned}$\end{center}

\noindent
In $(*)$, we used the Orbit-Stabilizer theorem.
\noindent
By comparing (1) and (2), we obtain
\[|\pi_0{(\mathcal{G}\ltimes X)}|=\frac{1}{|(\Iso\;\mathcal{G})_1|}\sum_{g\in(\Iso\;\mathcal{G})_1}|\Fix(g)|.\]
\end{proof}

\begin{definition}
If $\mathcal{H}$ and $\mathcal{K}$ are subgroupoids of $\mathcal{G}$, then the functor $X_{\mathcal{H},\mathcal{K}}:\mathcal{H}\times\mathcal{K}\to \set$ is defined by the following conditions: 
\begin{enumerate}
\item For an object $(s,t)$, define $X_{\mathcal{H},\mathcal{K}}((s,t)):=\mathcal{G}(t,s)$.
\item For a morphism $(h,k):(s,t)\to (s',t')$, define $X_{\mathcal{H},\mathcal{K}}((h,k)):\mathcal{G}(t,s)\to
\mathcal{G}(t',s')$ by $g\mapsto hgk^{-1}$.
\end{enumerate}
\end{definition}

\begin{proposition}
If $\mathcal{H}$ and $\mathcal{K}$ are subgroupoids of $\mathcal{G}$, then \[\mathcal{G}\ltimes X_{\mathcal{H},\mathcal{K}}\cong\incl^\mathcal{G}_\mathcal{K}\downarrow \incl^\mathcal{G}_\mathcal{H}.\]
\end{proposition}

\begin{proof}
This follows from the fact that the connected component of 
$\incl^\mathcal{G}_\mathcal{K}\downarrow \incl^\mathcal{G}_\mathcal{H}$ containing $g$  can be identified with $\mathcal{H}g\mathcal{K}$.
\end{proof}

The idea of relating double cosets to comma categories was inspired by \cite{BD20}.

\begin{theorem}
Let $\mathcal{H}$ and $\mathcal{K}$ be wide subgroupoids of a connected groupoid $\mathcal{G}$.
Let $\mathcal{H}\cong\coprod_{\lambda} \mathcal{H}_{\lambda}$ and $\mathcal{K}\cong\coprod_{\mu} \mathcal{K}_{\mu}$ be their decompositions into connected components.
For $g\in\mathcal{G}_1$, let $\mathcal{H}_\lambda$ and $\mathcal{K}_\mu$ be the connected components containing  $\cod g$ and $\dom g$, respectively. Then,
\[|\mathcal{H}g\mathcal{K}|=\frac{|(\mathcal{H}_\lambda)_0||(\mathcal{K}_\mu)_0||\mathcal{H}_{\cod g}||\mathcal{K}_{\dom g}|}{|\mathcal{H}_{\cod g}\cap{}^g\mathcal{K}_{\dom g}|}.\]
holds.
\end{theorem}

\begin{proof}
It suffices to count the number of $\Mor_\mathcal{H}(\cod g,-)$ and $\Mor_\mathcal{K}(-,\dom g)$. Using Lemma \ref{re0}, these can be computed as follows:  
\[|\Mor_\mathcal{H}(\cod g,-)|=|(\mathcal{H}_\lambda)_0||\mathcal{H}_{\cod g}|,\]
\[|\Mor_\mathcal{K}(-,\dom g)|=|(\mathcal{K}_\mu)_0||\mathcal{K}_{\dom g}|.\]

Suppose that $h_1gk_1=h_2gk_2$ holds. Then, we have $h_2^{-1}h_1=gk_2k_1^{-1}g^{-1}$, where the left-hand side belongs to $\mathcal{H}_{\cod g}$ and the right-side belongs to ${}^g\mathcal{K}_{\dom g}$. Therefore, by dividing the total number of pairs $|(\mathcal{H}_\lambda)_0||(\mathcal{K}_\mu)_0||\mathcal{H}_{\cod g}||\mathcal{K}_{\dom g}|$ by $|\mathcal{H}_{\cod g}\cap {}^g\mathcal{K}_{\dom g}|$, we obtain 
\[|\mathcal{H}g\mathcal{K}|=\frac{|(\mathcal{H}_\lambda)_0||(\mathcal{K}_\mu)_0||\mathcal{H}_{\cod g}||\mathcal{K}_{\dom g}|}{|\mathcal{H}_{\cod g}\cap{}^g\mathcal{K}_{\dom g}|}.\]
\end{proof}

\section{Induced representations}
In this section, we prove basic properties the induction of groupoid actions and the induction of linear representations. 
\begin{lemma}
Let $\triangle:\set\to\set^\mathcal{H}$ be the diagonal functor and $1$ be a terminal object of $\set$. Then \[\mathcal{G}/\mathcal{H}\cong \Ind^\mathcal{G}_\mathcal{H} \triangle 1.\]
\end{lemma}

\begin{proof}
Since we have a diagram 
\[\begin{tikzcd}
1 \arrow[r,"G"] & \mathcal{G} \arrow[dr, "\Ind^\mathcal{G}_\mathcal{H}\triangle 1"{name=U}] &{} &{}\\
\incl^\mathcal{G}_\mathcal{H}\downarrow G \arrow[u,"Q"] \arrow[r,"P"'] & \mathcal{H}  \arrow[ul,,shorten >=15pt,,shorten <=15pt,"\kappa",Rightarrow]\arrow[u, "\incl^\mathcal{G}_\mathcal{H}"]   \arrow[r, "\triangle 1"'{name=V}] & \set 
\ar[Rightarrow,dashed,shorten >=15pt,xshift=-15pt,yshift=5pt,"\tau"',from=V,to=U]
\end{tikzcd}
\]
by the pointwise Kan extension, we obtain
\[\begin{aligned}\Ind^\mathcal{G}_\mathcal{H}\triangle 1(G)&\cong\colim(\incl^\mathcal{G}_\mathcal{H}\downarrow G\xrightarrow{P}\mathcal{H}\xrightarrow{\triangle 1}\set)\\
&\cong\coprod_{l\in(\incl^\mathcal{G}_\mathcal{H}\,\downarrow\,G)_0}(\triangle 1\circ P)(l)/R \end{aligned},\]
where $R$ is the binary relation defined as follows:
\[\langle l,x\rangle\;R\;\langle k,y\rangle \iff h:\dom l\to \dom k \;\text{such that}\;(\triangle 1\circ P)(h)(x)=y\;\text{holds}.\]
Since $\incl^{\mathcal{G}}_{\mathcal{H}}\downarrow G$ is a groupoid, $R$ becomes an equivalence relation.
Let $l:H\to G$, $k:H'\to G$ and $h:H\to H'$. By commutativity, the equality $l=kh$ holds. From this, we see that $l$ and $k$ belong to the same left congruence.
Since $(\triangle 1\circ P)(h)=\triangle 1(h)=\id_1$, the equality $(\triangle 1\circ P)(h)(x)=y$ implies $x=y$.
By combining the results, $\Ind^\mathcal{G}_\mathcal{H}\triangle1(G)$ can be identified with the disjoint union of copies of $\{\bullet\}$ indexed by a complete set of the representatives $\mathcal{G}/\mathcal{H}(G)$.
\end{proof}

\begin{theorem}\label{ind}
Let $\mathcal{H}$ be a connected subgroupoid of $\mathcal{G}$. Then $\mathbb{C}[\mathcal{G}/\mathcal{H}]\cong \Ind^\mathcal{G}_\mathcal{H}\Tri$.  
\end{theorem}

\begin{proof}
By Lemma \ref{re1}, we have an isomorphism 
\[\Ind^\mathcal{G}_\mathcal{H}\Tri\cong\Ind^\mathcal{G}_\mathcal{H}(F\circ\triangle 1).\]
From the definition of induction and Lemma \ref{re1}, we obtain the following diagram.
\[\begin{tikzcd}
\mathcal{G} \arrow[dr, "\Ind^\mathcal{G}_\mathcal{H}\ \triangle1"{name=U}]\arrow[drr,bend left=30, "\Ind^\mathcal{G}_\mathcal{H}\ \Tri"] &{} &{}\\
\mathcal{H}\arrow[u, "\incl^\mathcal{G}_\mathcal{H}"]   \arrow[r, "\triangle 1"'{name=V}] & \set 
\ar[Rightarrow,dashed,shorten >=20pt,xshift=-15pt,yshift=5pt,"\tau"',from=V,to=U]
\arrow[r, "F"'] 
&\vect
\end{tikzcd}
\]
Since $F$ preserves finite coproducts, it commutes with Kan extensions. Therefore it follows that $\mathbb{C}[\mathcal{G}/\mathcal{H}]\cong \Ind^\mathcal{G}_\mathcal{H}\Tri$.
\end{proof}

\section{Groupoid action on functions}
In this section, we develop an approach to double cosets using linear representations of groupoids.
\begin{theorem}\label{dim}
Let \[\mathbbm{C}_{\mathcal{H}\backslash\mathcal{G}/\mathcal{K}}:=\{\phi:\mathcal{G}_1\to\mathbb{C}\mid \phi(h^{-1}gk)=\phi(g)\,,\,\forall (h,g,k)\in {\mathcal{H}_1}\,_{\cod}\kern-0.1em\times_{\cod}\,{\mathcal{G}_1} {}_{\dom}\kern-0.1em\times_{\cod}\,\mathcal{K}_1\}.\]  Then 
\[\dim \mathbbm{C}_{\mathcal{H}\backslash\mathcal{G}/\mathcal{K}}=|\mathcal{H}\backslash\mathcal{G}/\mathcal{K}|.\]
\end{theorem}

\begin{proof}
Let $x=hyk$ and $\phi\in \mathbbm{C}_{\mathcal{H}\backslash\mathcal{G}/\mathcal{K}}$. Then
\[\begin{aligned}
\phi(x)&=\phi(hyk)\\
&=\phi((h^{-1})^{-1}yk)\\
&=\phi(y)
\end{aligned}\]holds.
Thus, we see that the set $\mathbbm{C}_{\mathcal{H}\backslash\mathcal{G}/\mathcal{K}}$ consists of all functions that take constant values on each double coset. Therefore, we have $\dim \mathbbm{C}_{\mathcal{H}\backslash\mathcal{G}/\mathcal{K}}=|\mathcal{H}\backslash\mathcal{G}/\mathcal{K}|$.
\end{proof}

\begin{definition}
Let $\mathcal{H}$ be a connected wide subgroupoid of $\mathcal{G}$. Define the functor $Y_{\mathcal{H}}:\mathcal{G}\to \vect$ as follows: 

\noindent
$(1)$ For an object $G$, define 
\[Y_{\mathcal{H}}(G):=\{\phi:\mathcal{G}(-,G)\to\mathbb{C}\mid \phi(gh)=\phi(g),\;\forall(g,h)\in \mathcal{G}_1 {}_{\dom}\kern-0.4em\times_{\cod}\,\mathcal{H}_1\},\] 
$(2)$ For a morphism $g$, define $Y_{\mathcal{H}}(g):Y_{\mathcal{H}}(G)\to
Y_{\mathcal{H}}(G')$ by \[f\mapsto f(g^{-1}-).\]
\end{definition}

\begin{theorem}\label{R}
Let $\mathcal{H}$ be a connected wide subgroupoid of $\mathcal{G}$. Then \[\mathbb{C}[\mathcal{G}/\mathcal{H}]\cong Y_{\mathcal{H}}.\]
\end{theorem}

\begin{proof}
For each $G\in\mathcal{G}_0$, define maps as follows: 
\[\begin{array}{rccc}
\theta_G:&\mathbb{C}[\mathcal{G}/\mathcal{H}](G)&\longrightarrow&Y_{\mathcal{H}}(G)\\
 &\rotatebox{90}{$\in$}& &\rotatebox{90}{$\in$}\\
 &g\mathcal{H}&\longmapsto&\delta_{g\mathcal{H}}
\end{array}
\]

\noindent
where, 
\[\delta_{g\mathcal{H}}(k)=
\begin{cases*}
1& \text{if} $k\in g\mathcal{H}_1$,\\
0& \text{otherwise}.
\end{cases*}
\]

For each $G\in\mathcal{G}_0$ , $\theta_G$ is an isomorphism of vector spaces and for each $u:G\to G'$, the diagram. 

\[\xymatrix{
\mathbb{C}[\mathcal{G}/\mathcal{H}](G)\ar[d]_-{\theta_G}\ar[rr]^-{\mathbb{C}[\mathcal{G}/\mathcal{H}](u)}&&\mathbb{C}[\mathcal{G}/\mathcal{H}](G')\ar[d]^-{\theta_{G'}}\\
Y_{\mathcal{H}}(G)\ar[rr]_-{Y_{\mathcal{H}}(u)}&&Y_{\mathcal{H}}(G')}\]
\noindent
defined by
\[\xymatrix{
g\mathcal{H}\ar@{|->}[d]\ar@{|->}[r]&ug\mathcal{H}\ar@{|->}[d]\\
\delta_{g\mathcal{H}}\ar@{|->}[r]&\delta_{ug\mathcal{H}}}\]
\noindent
is commutative. Hence $\theta:=\{\theta_G\}_{G\in\mathcal{G}_0}$ is a natural transformation from $\mathbb{C}[\mathcal{G}/\mathcal{H}]$ to $Y_{\mathcal{H}}$, and we obtain 
\[\mathbb{C}[\mathcal{G}/\mathcal{H}]\cong Y_{\mathcal{H}}.\] 
\end{proof}

\begin{theorem}\label{hom}
Let $\mathcal{H}$ and $\mathcal{K}$ be connected wide subgroupoids of $\mathcal{G}$. Then there exists an isomorphism of $\mathbb{C}$-vector spaces \[\mathbbm{C}_{\mathcal{H}\backslash\mathcal{G}/\mathcal{K}}\cong \vect^\mathcal{G}(Y_\mathcal{H},Y_\mathcal{K}).\]  
\end{theorem}

\begin{proof}
For any $\phi\in\mathbb{C}_{\mathcal{H}\backslash\mathcal{G}/\mathcal{K}}$ and any $G\in\mathcal{G}_0$, we define the map as follows:

\[\begin{array}{rccc}
\theta_G:&Y_\mathcal{H}(G)&\longrightarrow &
Y_\mathcal{K}(G)\\
 &\rotatebox{90}{$\in$}& &\rotatebox{90}{$\in$}\\
 &f&\longmapsto&
 (g\mapsto \frac{1}{|(\Iso\,\mathcal{G})_1|}\sum_{x\in\Mor_{\mathcal{G}}(-,\cod g)}\phi(x^{-1}g)f(x) ).\\
\end{array}
\]

We define the maps $S$ and $T$ as follows: 
\[\begin{array}{rccc}
S:&\mathbbm{C}_{\mathcal{H}\backslash\mathcal{G}/\mathcal{K}}&\longrightarrow&\vect^\mathcal{G}(Y_\mathcal{H},Y_\mathcal{K})\\
 &\rotatebox{90}{$\in$}& &\rotatebox{90}{$\in$}\\
 &\phi&\longmapsto&\{\theta_G\}_{G\in\mathcal{G}_0},
\end{array}
\]

\[\begin{array}{rccc}
T:&\vect^\mathcal{G}(Y_\mathcal{H},Y_\mathcal{K})&\longrightarrow &
\mathbb{C}_{\mathcal{H}\backslash\mathcal{G}/\mathcal{K}} \\
 &\rotatebox{90}{$\in$}& &\rotatebox{90}{$\in$}\\
 &\{\psi_G\}_{G\in\mathcal{G}_0}&\longmapsto&
 T(\{\psi_G\}_{G\in\mathcal{G}_0}).
\end{array}
\]

\noindent
Here, $T(\{\psi_G\}_{G\in\mathcal{G}_0})$ is the map induced by the universal property of coproduct associated with the family $\{\psi_G(\delta_G)\}_{G\in\mathcal{G}_0}$ where $\delta_G$ is defined by
\[\delta_G(g)=\begin{cases*}
\frac{|\mathcal{G}_G|}{|\mathcal{H}_G|}& \text{if} $g\in \mathcal{H}_1$,\\
0& \text{otherwise}.
\end{cases*}
\]

We write $g*f:=Y_{\mathcal{H}}(g)(f)$ for a morphism $g$ and an element $f\in Y_{\mathcal{H}}(\dom g)$. Similarly, we write $g\bullet f:=Y_{\mathcal{K}}(g)(f)$ for a morphism $g$ and an element $f\in Y_{\mathcal{K}}(\dom g)$.\\

(i) Well-definedness of $S$

For $(g,k)\in {\mathcal{G}_1} \,_{\dom}\times_{\cod}\,\mathcal{K}_1$,
 we obtain

$\begin{aligned}
S(\phi)(f)(gk)&=\frac{1}{|(\Iso\,\mathcal{G})_1|}\sum_{x\in\Mor_{\mathcal{G}}(-,\cod g)}\phi(x^{-1}gk)f(x)\\
&=\frac{1}{|(\Iso\,\mathcal{G})_1|}\sum_{x\in\Mor_{\mathcal{G}}(-,\cod g)}\phi(x^{-1}g)f(x)\\
&=S(\phi)(f)(g).\\
\end{aligned}$

For each $u:G\to G'$, we obtain

$\begin{aligned}
(u\bullet \theta_G(f))(g)&=\frac{1}{|(\Iso\,\mathcal{G})_1|}\sum_{x\in\Mor_{\mathcal{G}}(-,\cod u^{-1}g)}\phi(x^{-1}u^{-1}g)f(x)\\
&=\frac{1}{|(\Iso\,\mathcal{G})_1|}\sum_{z\in\Mor_{\mathcal{G}}(-,\cod g)}\phi(z^{-1}g)f(u^{-1}z)\\
&=\theta_{G'}(u*f)(g).
\end{aligned}$

Therefore, $\{\theta_G\}_{G\in\mathcal{G}_0}\in \vect^\mathcal{G}(Y_\mathcal{H},Y_\mathcal{K})$.\\

(ii) Well-definedness of $T$

For $(h,g,k)\in {\mathcal{H}_1}\,_{\cod}\times_{\cod}\,{\mathcal{G}_1} \,_{\dom}\times_{\cod}\,\mathcal{K}_1$,
 we obtain
 
$\begin{aligned}
\psi_{\cod g}(\delta_{\cod g})(g)&=\psi_{\cod g}(\delta_{\cod g})(gk)\\
&=\psi_{\cod g}(h*\delta_{\dom h})(gk)\\
&=(h\bullet\,\psi_{\dom h}(\delta_{\dom h}))(gk)\\
&=\psi_{\dom h}(\delta_{\dom h})(h^{-1}gk).
\end{aligned}$

Therefore, $T(\{\psi_G\}_{G\in\mathcal{G}})\in\mathbb{C}_{\mathcal{H}\backslash\mathcal{G}/\mathcal{K}}$.
\\

(iii) $T\circ S=\id_{\mathbb{C}_{\mathcal{H}\backslash\mathcal{G}/\mathcal{K}}}$

For each $\phi\in\mathbb{C}_{\mathcal{H}\backslash\mathcal{G}/\mathcal{K}}$,

$\begin{aligned}
(T\circ S)(\phi)(g)&=\frac{1}{|(\Iso\,\mathcal{G})_1|}\sum_{x\in\Mor_{\mathcal{G}}(-,\cod g)}\phi(x^{-1}g) \delta_{\cod g}(x)\\
&=\frac{1}{|(\Iso\,\mathcal{H})_1|}\sum_{x\in\Mor_{\mathcal{H}}(-,\cod g)}\phi(x^{-1}g)\\
&=\frac{1}{|(\Iso\,\mathcal{H})_1|}\sum_{x\in\Mor_{\mathcal{H}}(-,\cod g)}\phi(g)\\
&=\phi(g).
\end{aligned}$\\

Therefore, $T\circ S=\id_{\mathbb{C}_{\mathcal{H}\backslash\mathcal{G}/\mathcal{K}}}$.\\

(iv) $S\circ T=\id_{\vect^{\mathcal{G}}(Y_\mathcal{H},Y_\mathcal{K})}$

$Y_\mathcal{H}(G)$ is generated by $\{k_1*\delta_{\dom k_1},\ldots,k_n*\delta_{\dom k_n}\}$ where $\{k_1,\ldots ,k_n\}$ is a complete set of the representatives of $\Mor_\mathcal{G}(-,G)/\sim_{\mathcal{H},G}$.

For each $\psi\in\vect^\mathcal{G}(Y_\mathcal{H},Y_\mathcal{K})$, we obtain

$\begin{aligned}
 (S\circ T)(\psi)_{\cod g}(\delta_{\cod g}) (g)
 &=\frac{1}{|(\Iso\,\mathcal{G})_1|}\sum_{x\in\Mor_{\mathcal{G}}(-,\cod g)}\psi_{\dom x}(\delta_{\dom x})(x^{-1}g)\delta_{\cod g}(x)\\
 &=\frac{1}{|(\Iso\,\mathcal{G})_1|}\sum_{x\in\Mor_{\mathcal{G}}(-,\cod g)}(x\bullet\psi_{\dom x}(\delta_{\dom x}))(g)\delta_{\cod g}(x)\\
 &=\frac{1}{|(\Iso\,\mathcal{G})_1|}\sum_{x\in\Mor_{\mathcal{G}}(-,\cod g)}\psi_{\cod g}(x*\delta_{\dom x})(g)\delta_{\cod g}(x)\\
 &=\frac{1}{|(\Iso\,\mathcal{G})_1|}\sum_{x\in\Mor_{\mathcal{H}}(-,\cod g)}\psi_{\cod g}(\delta_{\cod g})(g)\delta_{\cod g}(x)\\
 &=\frac{1}{|(\Iso\,\mathcal{H})_1|}\sum_{x\in\Mor_{\mathcal{H}}(-,\cod g)}\psi_{\cod g}(\delta_{\cod g})(g)\\
 &=\psi_{\cod g}(\delta_{\cod g})(g).
\end{aligned}$

Hence, 

$\begin{aligned}
 (S\circ T)(\psi)_{\cod k_i}(k_i*\delta_{\dom k_i}) &=k_i\bullet(S\circ T)(\psi)_{\dom k_i}(\delta_{\dom k_i}) \\
 &=k_i\bullet\psi_{\dom k_i}(\delta_{\dom k_i})\\
 &=\psi_{\cod k_i}(k_i*\delta_{\dom k_i}).
\end{aligned}$

holds, and therefore $S\circ T=\id_{\vect^{\mathcal{G}}(Y_\mathcal{H},Y_\mathcal{K})}$.\\

From (i) $\sim$ (iv), we obtain
\[\mathbb{C}_{\mathcal{H}\backslash\mathcal{G}/\mathcal{K}}\cong \vect^\mathcal{G}(Y_\mathcal{H},Y_\mathcal{K}).\].

\end{proof}

\begin{corollary}
Let $\mathcal{H}$ and $\mathcal{K}$ be connected wide subgroupoids of $\mathcal{G}$. 
Then, 
\[
|\mathcal{H}\backslash\mathcal{G}/\mathcal{K}|=\langle \chi_{\Ind^\mathcal{G}_\mathcal{H}\Tri},\chi_{\Ind^\mathcal{G}_\mathcal{K}\Tri} \rangle\]
holds.
\end{corollary}

\begin{proof}
\[\begin{aligned}
|\mathcal{H}\backslash\mathcal{G}/\mathcal{K}|&=\dim (\vect^\mathcal{G}(Y_\mathcal{H},Y_\mathcal{K}))&\text{(By Theorem \ref{dim} and Theorem \ref{hom})}\\
&= \dim (\vect^\mathcal{G}(\mathbb{C}[\mathcal{G}/\mathcal{H}],\mathbb{C}[\mathcal{G}/\mathcal{K}]))&\text{(By Theorem \ref{R})}\\
&=\dim (\vect^\mathcal{G}(\Ind^\mathcal{G}_\mathcal{H}\Tri,\Ind^\mathcal{G}_\mathcal{K}\Tri))&\text{(By Theorem \ref{ind})}\\
&=\langle \chi_{\Ind^\mathcal{G}_\mathcal{H}\Tri},\chi_{\Ind^\mathcal{G}_\mathcal{K}\Tri} \rangle&\text{(By Corollary \ref{product}).}
\end{aligned}\]

\end{proof}

\section*{Acknowledgement}
The author sincerely thanks his supervisor, Professor Fumihito Oda, for his support and valuable comments.


\begin{thebibliography}{99}
\bibitem[AM20]{AM20} J. \'{A}vila; V. Mar\'{i}n, The Notions of Center, Commutator and Inner Isomorphism for Groupoids, Ingenier\'{i}a y Ciencia, 16(31)(2020), 7-26.
\bibitem[BD20]{BD20} P. Balmer and I. Dell'Ambrogio, Mackey 2-functors and Mackey 2-motives (European Mathmatical'Society (EMS), Z\"{u}rich, 2020).
\bibitem[BE19]{BE19} J. J. Barbar\'{a}n S\'{a}nchez and L. EI Kaoutit, Linear Representations and Frobenius Morphisms of Groupoids, SIGMA, \textbf{15} (2019), 019, 33pp.
\bibitem[BGLPT23]{BGLPT23} G. Beier, C. Garcia,W. G. Lautenschlaeger, J. Pedrotti and T. Tamusiunas. Generalizations of Lagrange and Sylow theorems for groupoids, S$\tilde{a}$o Paulo J. Math. pages 1-20, 2023.
\bibitem[Bra27]{Bra27} H. Brandt, \"{U}ber eine Verallgemeinerung des Gruppenbegriffes, Math. Ann., \textbf{96} (1927), 360-366.
\bibitem[IR19]{IR19} A. Ibort, M. A. Rodr\'{i}guez. On the structure of finite groupoids and their representations. Symmetry, \textbf{11}, 414 (2019).
\bibitem[Iva02]{Iva02} G. Ivan, Algebraic constructions of Brandt Groupoids, Proceeding of the Algebra Symposium, Babes-Bolyai University Cluj, 69-90, 2002.
\bibitem[MP22]{MP22} V. Mar\'{i}n, H. Pinedo, Groupoids: Direct products, semidirect products and solvability, Algebra and Discrete Math., 33(2), (2022), 92-107. Doi:10.12958/adm 1772.
\end{thebibliography}
\end{document}